\documentclass{amsart}

\usepackage{latexsym}
\usepackage{enumerate}
\usepackage{amssymb}
\usepackage{amsfonts}
\usepackage{amsmath}
\usepackage{mathrsfs}
\usepackage{amsthm}
\usepackage{geometry}
\usepackage[colorlinks,
linkcolor=red,
anchorcolor=blue,
citecolor=green
]{hyperref}
\usepackage{varioref}
\usepackage{hyperref}
\usepackage{cleveref}
\usepackage[all]{xy}
\usepackage{tikz}

\newcommand{\Ric}{\mathrm{Ric}}

\newcommand{\Vol}{\mathrm{Vol}}

\newcommand{\inj}{\mathrm{inj}}
\newcommand{\sq}{\backslash}

\newcommand{\eps}{\epsilon}
\newcommand{\Id}{\text{Id}}

\newcommand{\dist}{\mathrm{dist}}

\newcommand{\sphere}{{\mathbb{S}}}

\newcommand{\inte}{\mathrm{int}}

\newtheorem{thm}{Theorem}[section]

\newtheorem{fact}{Fact}[section]
\newtheorem{prop}{Proposition}[section]
\newtheorem{coro}{Corollary}[section]
\newtheorem{lmm}{Lemma}[section]

\newtheorem{pbm}{Problem}[section]

\newtheorem*{inductionhyp}{Inductive Hypothesis}

\theoremstyle{remark} 
\newtheorem*{rmk}{Remark}

\title{Examples of Ricci limit spaces with infinite holes}

\author{Shengxuan Zhou}
\address{Beijing International Center for Mathematical Research\\
Peking University\\
Beijing\\ 100871\\ China}
\email{zhoushx19@pku.edu.cn}

\thanks{}
\date{}
\dedicatory{}

\begin{document}

\begin{abstract}
Let $n\geq 3$, $\lambda \in \mathbb{R} $, and $(X,h)$ be an $n$-dimensional smooth complete Riemannian manifold with $\Ric_h > \lambda $. 

In this paper, we construct, for each given $\epsilon >0$, a sequence of $(n+2)$-dimensional manifolds $(M_{i} ,g_i ) \stackrel{GH}{\longrightarrow} (X_\epsilon ,d_\epsilon ) $ with $\Ric_{g_i} > \lambda $, such that $d_{GH} (X,X_\epsilon ) \leq \epsilon $, and $X_\epsilon $ is homeomorphic to the space obtained by removing an infinite number of balls from $X$. Hence $X_\epsilon $ has dense boundary with an infinite number of connected components. Moreover, $X_\epsilon $ has no open subset which is topologically a manifold. This generalizes Hupp-Naber-Wang's result \cite{hnw1} from $4$-dimensional case to the case of dimension $n\geq 3$. Our approach avoids the use of local complex surface structures, and as a corollary resolves an open problem of Naber \cite[Open Problem 3.4]{naber1} in dimension $3$.
\end{abstract}
\maketitle
\tableofcontents

\section{Introduction}
\label{introduction}

In the theory of Gromov-Hausdorff limits, a significant focus lies in describing the geometry and topology of measured Gromov-Hausdorff limit spaces as following:
\begin{eqnarray}
    (M_{i} ,g_i ,\nu_i ,p_i ) \stackrel{GH}{\longrightarrow} (X^n ,d ,\nu ,p ) ,\;\; \Ric_{g_i} \geq -\lambda ,\;\; \nu_i = \frac{1}{\Vol (B_1 (p_i)) } d\nu_{g_i },
\end{eqnarray}
where $n\in\mathbb{N}$ is the rectifiable dimension of $(X ,d ,\nu )$, and all $ M_i $ have the same dimension. Note that Colding-Naber \cite{coldingnaber1} proved that the limit $(X ,d ,\nu )$ is $n$-rectifiable for a unique $n $.

In the non-collapsed case, which is to say $\Vol (B_1 (p_i)) \geq v >0 $, the works of Colding \cite[Theorem 0.1]{colding1} and Cheeger-Colding \cite[Theorem 5.9]{chco3} show that the measure $ \nu = c \mathcal{H}^n $ for some $c>0$, and the set of $n$-regular points $R_n (X) \subset X$ is a subset of full-$\nu $ measure, where $\mathcal{H}^n $ is the $n$-dimensional Hausdorff measure on $(X^n ,d )$, and a point $x\in X$ is called a $k$-regular point if and only if every tangent cone of $X$ at $x$ is isometric to $\mathbb{R}^k$. Then a Reifenberg type result of Cheeger-Colding \cite[Theorem 5.14]{chco3} shows that the limit space has a manifold structure away from a codimension $2$ set, which is the complement of an open neighborhood of $R_n (X)$. More recently, the work of Cheeger-Jiang-Naber \cite[Theorem 1.14]{jcwsjan1} was able to prove that in this case the codimension $2$ Hausdorff measure of the subset we removed can always be assumed finite. 

In the collapsed case, which is to say $\lim\limits_{i\to\infty} \Vol (B_1 (p_i)) = 0 $, Colding-Naber \cite{coldingnaber1} proved that the set of $n$-regular points $R_n (X) \subset X$ is a subset of full-$\nu $ measure. However, Pan-Wei's example \cite{panwei1} has shown that the Hausdorff dimension of $X\sq R_n (X) $ can be greater than the Hausdorff dimension of $ R_n (X) $. By using their almost $n$-splitting map, Cheeger-Colding \cite[Theorem 3.26]{chco5} proved that there exists a subset $U$ of $R_n (X)$ such that $\nu ( U ) >0 $ and $U$ is bi-Lipschitz to a subset of $\mathbb{R}^n$. However, in general, $U$ is not an open subset of $X$.

A natural question arises: if $x\in R_n (X)$, does there exist an open neighborhood of $x$ homeomorphic to an open subset of $\mathbb{R}^n$, similar to the non-collapsed case? This leads to the following well-known open problem:

\begin{pbm}[{\cite[Open Problem 3.4]{naber1}}]
\label{naberopenpbm}
Let $(M^{n+k}_{i} ,g_i ,\nu_i ,p_i ) \stackrel{GH}{\longrightarrow} (X^n ,d ,\nu ,p ) $ with $\Ric_{g_i} \geq -\lambda $ and $\Vol (B_1 (p_i)) \to 0 $, where $k\geq 0$, and $n $ is the rectifiable dimension of $X$. Does there exist an open subset of full-$\nu $ measure $R(X)\subset X$ which is homeomorphic to an $n$-manifold?
\end{pbm}

It's worth noting that for $n=1$, the answer to Problem \ref{naberopenpbm} is affirmative \cite{honda1, lina1}.

Recently, Hupp-Naber-Wang \cite{hnw1} showed that for any $4$-dimensional manifold with a Ricci lower bound, there exists a sequence of $6$-dimensional manifolds with a uniform Ricci lower bound such that the Gromov-Hausdorff limit is close to the given $4$-dimensional manifold, but the limit space has no open subset which is topologically a manifold. In topological sense, the Ricci limit space they constructed is the blowing up of a dense subset in the given $4$-dimensional manifold. As a corollary, by considering the product of a $4$-dimensional manifold and an $(n-4)$-dimensional manifold, they provided a negative answer of Problem \ref{naberopenpbm} when $n\geq 4$. 

However, Problem \ref{naberopenpbm} is still open for smaller $n$.

In this paper, we construct a sequence of $(n+2)$-dimensional manifolds for any given $n$-dimensional manifolds $(n\geq 3)$ with Ricci lower bound such that the limit space has no open subset which is a topological manifold. The following theorem is our main theorem in this paper.

\begin{thm}\label{thmnonmfd}
Let $(X , h)$ be an $n$-dimensional smooth complete boundary free Riemannian manifold with $n\geq 3$ and $\Ric_h > \lambda \in\mathbb{R} $. Then for each $\epsilon >0$, there exists a metric space $(X_\epsilon ,d_\epsilon )$ such that
\begin{enumerate}[(a).]
\item $d_{GH} (X,X_\epsilon ) \leq \epsilon $,  \\[-3mm]
\item $(X_\epsilon ,d_\epsilon ) $ is $n$-rectifiable, \\[-3mm]
\item There exists a sequence of $(n+2)$-dimensional boundary free complete manifolds $(M_{i} ,g_i ) \stackrel{GH}{\longrightarrow} (X_\epsilon ,d_\epsilon )$ with $\Ric_{g_i} > \lambda $,  \\[-3mm]
\item For any open subset $U$ of $X_\epsilon $, the local $(n-1)$-homotopy group ${\rm{im}} \left( \pi_{n-1} (U ) \to \pi_{n-1} (X_\epsilon ) \right) $ and local $(n-1)$-homology group ${\rm{im}} \left( H_{n-1} (U;\mathbb{Z}) \to H_{n-1} (X_\epsilon ;\mathbb{Z}) \right) $ are infinitely generated. Consequently, every open subset of $X_\epsilon $ is not a topological manifold.
\end{enumerate}
\end{thm}

\begin{rmk}
Actually, by replacing $ \mathbb{S}^2 $ with $ \mathbb{S}^k $ in our construction, $k\geq 2$, one can obtain in the similar way a sequence of $(n+k)$-dimensional manifolds $(M_{i} ,g_i ) \stackrel{GH}{\longrightarrow} (X_\epsilon ,d_\epsilon )$ with $\Ric_{g_i} > \lambda $ and $X_\epsilon $ has infinite holes.
\end{rmk}

\begin{rmk}
For such examples, the dimension condition $n\geq 3$ here is essential. If $n=2$, then (d) implies that the local fundamental group is infinitely generated. But the generalized Margulis lemma \cite{kw1, jkw1} shows that the local fundamental group is finitely generated.
\end{rmk}

Theorem \ref{thmnonmfd} refines Hupp-Naber-Wang's answer to Problem \ref{naberopenpbm}. In particular, for $n\geq 3$, Theorem \ref{thmnonmfd} implies that one can approximate any $n$-dimensional Riemannian manifold with lower Ricci bound by $n$-rectifiable Ricci limit spaces with no open subset homeomorphic to a manifold.

Topologically, the main difference between our construction and that of Hupp-Naber-Wang is that we use the plumbing gluing surgery to puncture holes on the given manifold $X$, instead using the blow-up surgery on the local complex structure of $X$. By using this surgery and careful analysis near the boundaries of the balls we removed, one can construct for the general case of dimension $n\geq 3$.

Now we consider the boundary of Ricci limit spaces. Let $(X^n ,d ,\nu ,p )$ be a Ricci limit space. Then for $0\leq k\leq n-1$, one can define
\begin{eqnarray}
\mathcal{S}^{k} = \bigg\lbrace x\in X: \textrm{ no tangent cone at $x$ splits off $\mathbb{R}^{k+1}$ } \bigg\rbrace.
\end{eqnarray}
Now we can define the reduced boundary of $X$ by
\begin{eqnarray}
\partial^* X = \mathcal{S}^{n-1} \sq \mathcal{S}^{n-2}.
\end{eqnarray}
When $X$ is a non-collapsed $\mathbf{RCD}$ space, there are two definitions of boundary. The definition of boundary given by De Philippis-Gigli \cite{dephi1}, $\partial_{\rm DPG} X$, is just the closure of the reduced boundary. The other definition of a boundary, provided by Kapovitch-Mondino \cite{km1}, denoted by $\partial_{\rm KM} X$, is defined inductively via tangent cones: a point $x \in X$ belongs to $\partial_{\rm KM} X$ if and only if there exists a tangent cone at $x$ which is a metric cone whose cross-section is an $(n-1)$-dimensional $\mathbf{RCD}$ space with non-empty $\partial_{\rm KM}$. See also \cite{bruenabersemola1} for more details about boundaries.

The motivation for studying the geometry of boundaries comes from the research about $\mathbf{RCD}$ spaces. Note that Cheeger-Colding \cite[Theorem 6.1]{chco3} proved that the boundary of non-collapsed Gromov-Hausdorff limits of boundary free manifolds is always empty. However, in the collapsed case, non-trivial boundaries may occur, and it is therefore natural to compare the two notions of boundary introduced above. An interesting question is whether these two kinds of boundaries are always equal. As a corollary of Theorem \ref{thmnonmfd}, our construction also gives examples where Kapovitch-Mondino's boundary and reduced boundary are both dense and contain an infinite number of connected components. Thus the two kinds of boundaries may be not equal in the collapsed case.

\begin{coro}\label{coroboundary}
Let $n\geq 3$. Then there exists a sequence of $(n+2)$-dimensional boundary free complete manifolds $(M_{i} ,g_i ) \stackrel{GH}{\longrightarrow} (X^n_\infty ,d_\infty )$ with $\Ric_{g_i} \geq -1 $, such that $\partial_{KM} X_\infty $ is both dense and contains an infinite number of connected components. Moreover, $\partial_{\rm DPG} X_\infty = X_\infty $.
\end{coro}

\begin{rmk}
We thank Professor Aaron Naber for pointing out this interesting corollary.
\end{rmk}

This paper is organized as follows. At first, we will give an outline of our construction in Section \ref{sectionoutline}. In Section \ref{sectionpreliminary}, we collect some preliminary results. Then we will construct the local model in Section \ref{sectionlocalmodel}. The construction of $(M_i ,g_i)$ is presented in Section \ref{holessection}. Finally, the proofs of Theorem \ref{thmnonmfd} and Corollary \ref{coroboundary} are contained in Section \ref{proofsection}.

\textbf{Acknowledgement.} The author wants to express his deep gratitude to Professor Gang Tian for constant encouragement. The author would also be grateful to Professor Aaron Naber, Professor Wenshuai Jiang and Professor Jianchun Chu for valuable conversations. Additionally, the author thanks Professor Xiaochun Rong, Professor Guofang Wei and Professor Xiping Zhu for their interests in this note. Furthermore, the author expresses gratitude to Professor Xin Fu and Professor Jiyuan Han for their warm hospitality during his visit to Westlake University. In particular, the author would like to thank Professor Daniele Semola for pointing out some errors in Introduction of the first edition.

\section{Outline of the construction}
\label{sectionoutline}

At first, we describe our construction in the topological sense. For the moment, let's set aside considerations of Ricci curvature and smoothness.

Our approach to puncturing holes in the manifold $X$ is to consider plumbing glue surgery on higher dimensional manifolds. Although the manifolds in the statement of Theorem \ref{thmnonmfd} are manifolds without boundary, the new boundaries will appear after we puncture holes on the manifold $X$. Thus we assume here that $X$ is a manifold with boundary in the following description.

Let $n,k\geq 1$, $X$ be an $n$-dimensional manifold with boundary, $M$ be an $(n+k-1)$-dimensional manifold without boundary, and $F:M\to X$ be a surjective continuous map such that:
\begin{itemize}
\item The restriction of $F$ on $F^{-1} ( \inte (X) ) \subset M $ gives a $\sphere^{k-1}$-bundle structure on $ F^{-1} ( \inte (X) )$.
\item The pre-image $F^{-1}(x) $ is a single point if $x\in \partial X$.
\end{itemize}

Choosing a pre-compact open subset $U$ of $\inte (X) $ such that $\bar{U}$ is homeomorphic to $\bar{B}^{n}_1 (0) \subset \mathbb{R}^n $. Then $\partial F^{-1} ( \bar{U} ) $ is homeomorphic to $\sphere^{n-1} \times \sphere^{k-1} $. Clearly, $\partial \left( \sphere^{n-1} \times \bar{B}^{k}_1 (0) \right)$ is also homeomorphic to $\sphere^{n-1} \times \sphere^{k-1} $. Hence we can get a new $(n+k-1)$-dimensional manifold $M'$ by gluing $\sphere^{n-1} \times \bar{B}^{k}_1 (0) $ to $M\sq F^{-1} ( U )$.

Now we consider the Euclidean distance function $d (y)= d_{\rm Euc} (y,0) $ on $\bar{B}^{k}_1 (0)$. Clearly, $d: \bar{B}^{k}_1 (0) \to [0,1]$ gives a $\sphere^{k-1}$-bundle structure on $ d^{-1} ( (0,1] ) = \bar{B}^{k}_1 (0) \sq \{ 0 \} $, and $d^{-1} (0) $ is a single point $\{ 0 \}$. Then we can get a map $\pi_d : \sphere^{n-1} \times \bar{B}^{k}_1 (0) \to \sphere^{n-1} \times [0,1] $ such that $\pi_d^{-1} $ gives a $\sphere^{k-1}$-bundle structure on $ \pi_d^{-1} ( \sphere^{n-1} \times (0,1] ) = \sphere^{n-1} \times \left( \bar{B}^{k}_1 (0) \sq \{ 0 \} \right) $, and $\pi_d^{-1} (y) $ is a single point if $y\in \sphere^{n-1} \times \{ 0 \} $. Then we can obtain an $n$-dimensional manifold $X'$ with boundary $\partial X' =\partial X \bigsqcup \sphere^{n-1} $ by gluing $\sphere^{n-1} \times [0,1] $ to $X\sq U$, and get a continuous map $F' : M' \to X'$ by gluing $F$ and $\pi_d$.

Actually, the metric on $\bar{B}^k_{1} (0)$ we used is isometric to the cone metric on $C(\sphere^{k-1}_\theta )$ for some $\theta\in (0,1)$. We then shrink the restriction of the metric on the $\sphere^{k-1}$-fibers. As the constant $\theta$ decreases, $C(\sphere^{k-1}_\theta )$ converges to the interval $[0,1]$. 

It is easy to see that the data $(X', M', F')$ also satisfies the conditions above, allowing us to replace $(X,M,F)$ by $ (X',M',F') $. By repeating this process, we can puncture an infinite number of holes in $X$, and the topology of the Ricci limit space that we have constructed is the space obtained by removing an infinite number of balls from $X$.

To make our construction easier to understand, we give here a graph for the case where $n=2$, $k=1$, $X_0=\bar{\mathbb{D}}_1 (0)$, $M_0 =\sphere^2 $ and $F_0$ is the standard projection. Here, $M_i$ is getting squashed, causing the Gromov-Hausdorff distance between $M_i $ and $X_i$ to decrease.

\smallskip

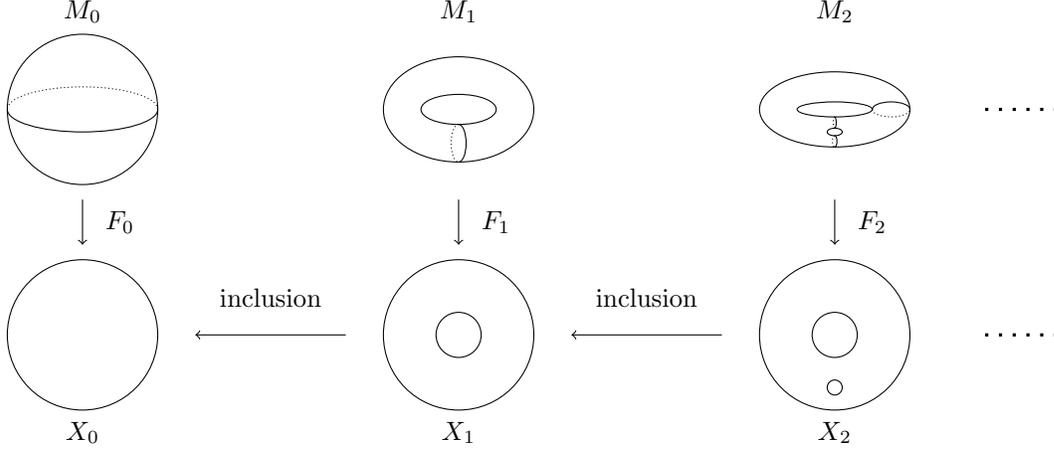
\begin{tikzpicture}
\draw (0,0) circle (1);
\draw (5,0) ellipse (1 and 0.7);
\draw (10,0) ellipse (1 and 0.5);
\draw (0,-3) circle (1);
\draw (5,-3) circle (1);
\draw (10,-3) circle (1);
\draw (-1,0) arc (180:360:1 and 0.3); 
\draw[densely dotted] (1,0) arc (0:180:1 and 0.3);
\draw (5,0) ellipse (0.5 and 0.2);
\draw (5,-0.7) arc (-90:90:0.1 and 0.25); 
\draw[densely dotted] (5,-0.2) arc (90:270:0.1 and 0.25);
\draw (10,0) ellipse (0.5 and 0.1);
\draw (10,-0.3) ellipse (0.1 and 0.05);
\draw[densely dotted] (10.5,0) arc (180:360:0.25 and 0.1); 
\draw (11,0) arc (0:180:0.25 and 0.1);
\draw (10,-0.5) arc (-90:90:0.03 and 0.075); 
\draw[densely dotted] (10,-0.35) arc (90:270:0.03 and 0.075);
\draw (10,-0.25) arc (-90:90:0.02 and 0.075); 
\draw[densely dotted] (10,-0.1) arc (90:270:0.02 and 0.075);
\draw (5,-3) circle (0.3);
\draw (10,-3) circle (0.3);
\draw (10,-3.7) circle (0.1);
\draw[->] (0,-1.2)--(0,-1.8);
\draw[->] (5,-1.2)--(5,-1.8);
\draw[->] (10,-1.2)--(10,-1.8);
\draw[->] (3.5,-3)--(1.5,-3);
\draw[->] (8.5,-3)--(6.5,-3);
\node at(0,1.3) {$M_0$};
\node at(5,1.3) {$M_1$};
\node at(10,1.3) {$M_2$};
\node at(0,-4.3) {$X_0$};
\node at(5,-4.3) {$X_1$};
\node at(10,-4.3) {$X_2$};
\node at(0.5,-1.5) {$F_0$};
\node at(5.5,-1.5) {$F_1$};
\node at(10.5,-1.5) {$F_2$};
\node at(2.5,-2.5) {inclusion};
\node at(7.5,-2.5) {inclusion};
\draw [very thick, loosely dotted] (12,0)--(13,0);
\draw [very thick, loosely dotted] (12,-3)--(13,-3);
\end{tikzpicture}

\smallskip

The above graph is to provide some intuition. Now we give an outline of our construction.

\begin{enumerate}[(I).]
    \item Write $(X_0,h_0) =(X,h) $, $(M_0 ,g_0 ) = \left( X\times \sphere^2  ,h+ (1+|\lambda|)^{-10} g_{\sphere^2} \right) $, $f_0 = (1+|\lambda|)^{-5} $, and $W_0 =\emptyset \subset M_0 $. Then we have $\Ric_{g_0} >\lambda $.
    \item Let $n\geq 3$, $(X_i,h_i)$ be an $n$-dimensional smooth Riemannian manifold with boundary, $f_i$ be a non-negative function on $X_i$, $f_i>0$ on $\inte (X_i) $, $(M_i,g_i)$ be an $(n+2)$-dimensional smooth Riemannian manifold without boundary, $\Ric_{g_i} >\lambda $, and $W_i\subset M_i $ be an $(n-1)$-dimensional submanifold such that $$ (M_i \sq W_i ,g_i ) \cong (\inte (X_i ) \times \sphere^2 ,h_i + f_i^2 g_{\sphere^2 } ) .$$
    \item Let $U_i $ be a small pre-compact open subset of $\inte (X_i )$, such that $\bar{U}_i $ diffeomorphic to $\bar{B}^n_1 (0) $. 
    \item Construct a metric $h_{i+1}$ and a non-negative function $f_{i+1}$ on $X_i \sq U_i$, such that $f_{i+1} =0 $ on $\partial U_i$, and the completion of $ ( (\inte (X_i ) \sq U_i ) \times \sphere^2 ,h_{i+1} + f_{i+1}^2 g_{\sphere^2 } ) $ near $\partial U_i$ is a smooth Riemannian manifold without boundary. Moreover, we can assume that $h_{i+1} =h_i $ and $f_{i+1} = \delta f_i $ outside a small neighborhood of $\partial U_i$.
    \item Consider the completion $(Y,d)$ of $ ( (\inte (X_i ) \sq U_i ) \times \sphere^2 ,h_{i+1} + f_{i+1}^2 g_{\sphere^2 } ) $. Then there is a natural embedding $W_i \to Y $, and tangent cone at any point in $W_{i }$ is always isometric to $C(\sphere^2_{\delta } ) \times \mathbb{R}^{n-1} $.
    \item Modify $f_{i+1}$ near $W_{i }$ such that the completion of $ ( (\inte (X_i ) \sq U_i ) \times \sphere^2 ,h_{i+1} + f_{i+1}^2 g_{\sphere^2 } ) $ near $\partial W_i $ becomes a smooth Riemannian manifold without boundary. See also Corollary \ref{corollarymodify2}.
    \item Let $(M_{i+1} ,g_{i+1} )$ be the completion of $ ( (\inte (X_i ) \sq U_i ) \times \sphere^2 ,h_{i+1} + f_{i+1}^2 g_{\sphere^2 } ) $, and $(X_{i+1},h_{i+1},f_{i+1})=(X_i\sq U_i ,h_i ,f_i)$. Then we do the process (ii)-(v) again and again.
    \item Since $d_{GH} (M_i ,X_i) = \Psi (i^{-1}) $, we see that the Gromov-Hausdorff limits $\lim M_i = \lim X_i =X_\epsilon $. By composing the inclusion maps $X_{i+1} = X_i \sq U_i \to X_i $, we can obtain a Lipschitz map $X_{\epsilon} \to X_i $. Then we can prove Theorem \ref{thmnonmfd}.
\end{enumerate}

\section{Preliminary}
\label{sectionpreliminary}

\subsection{Double warped products}
 
In this section, we recall the Ricci curvature of double warped products.
 
 Let $\varphi ,\phi $ be smooth nonnegative functions on $[0,\infty )$ such that $\varphi ,\phi $ are positive on $(0,\infty )$,
 \begin{eqnarray} \label{conditionphirho}
 \phi (0) >0 ,\; \phi^{\mathrm{(odd)}} (0) =0 , 
 \end{eqnarray}
and
\begin{eqnarray} \label{conditionvarphi}
\varphi (0) =0 ,\; \varphi' (0)= 1 ,\; \varphi^{\mathrm{(even)}} (0) =0 .
\end{eqnarray}
Then we can define a Riemannian metric on $ \mathbb{R}^{3} \times \sphere^n $ by
$$ g_{\varphi ,\phi } (r) = dr^2 + \varphi (r)^2 g_{\sphere^2} +  \phi (r)^2 g_{\sphere^n} .$$
See also \cite[Proposition 1.4.7]{pp1} for more details.

Let $X_0 =\frac{\partial }{\partial r} $, $X_1 \in T_{\sphere^2} $, and $X_2 \in T_{\sphere^n} $ be unit vectors. Then the Ricci curvature of $g_{\varphi ,\phi }$ can be expressed as following. 

\begin{lmm}\label{lmmriccicurvature}
Let $\varphi ,\phi $ and $ g_{\varphi , \phi }$ be as above. Then the Ricci curvature tensor of $ g_{\varphi , \phi }$ can be determined by
\begin{eqnarray}
\Ric_{g_{\varphi , \phi } } \left(X_0 \right) & = & - \left( 2 \frac{\varphi''}{\varphi} + n \frac{\phi''}{\phi} \right) X_0 ,\\
\Ric_{g_{\varphi , \phi } } \left(X_1 \right) & = & \left[-  \frac{\varphi''}{\varphi} + \frac{1-(\varphi')^2 }{\varphi^2} - n \frac{\varphi' \phi'}{\varphi \phi } \right] X_1 ,\\
\Ric_{g_{\varphi , \phi } } \left(X_2 \right) & = & \left[-  \frac{\phi''}{\phi} +(n-1) \frac{1-(\phi')^2 }{\phi^2} - 2 \frac{\varphi' \phi'}{\varphi \phi } \right] X_2 .
\end{eqnarray}
\end{lmm}

\begin{proof}
One can conclude it by a straightforward calculation. See also \cite[Subsection 4.2.4]{pp1}.
\end{proof}

Then we have the following gluing lemma for double warped products.

\begin{lmm}\label{lmmgluing}
Let $\varphi ,\phi $ and $ g_{\varphi , \phi }$ be as above. 

Assume that $\varphi' >0 , \phi'>0 $ on the interval $(r_0-2\delta , r_0 +2\delta ) \subset (0,\infty )$, and $\varphi , \phi$ are smooth on $(r_0-2\delta , r_0 ) \cup (r_0 , r_0 +2\delta )$. Suppose that $\varphi , \phi$ have smooth limits on both sides of $r_0$, $\varphi'_+ (r_0 ) \leq \varphi'_- (r_0 ) $, and $\phi'_+ (r_0 ) \leq \phi'_- (r_0 ) $. Then for any $\epsilon >0$, we can find smooth $\tilde{\varphi}$, $\tilde{\phi}$ such that $|\tilde{\varphi } -\varphi | + |\tilde{\phi} -\phi | \leq \epsilon $ on $(r_0 -2\delta ,r_0 +2 \delta )$, and $\tilde{\varphi} = \varphi $, $\tilde{\phi} = \phi $ outside $(r_0-\delta , r_0 +\delta )$, and $\inf \Ric_{g_{\tilde{\varphi} , \tilde{\phi} } } \geq \inf \Ric_{g_{\varphi , \phi } } - \epsilon $.
\end{lmm}

It follows from a $C^0$ gluing in the spirit of Perelman \cite[Section 4]{perel1}. See also \cite[Lemma 1.170]{menguy2} and \cite[Lemma 3.1]{hnw1}.

In general, if we consider the warped product $(M\times \sphere^n , g= g_M + f^2 g_{\sphere^n} )$, then the Ricci curvature operator can be given by
\begin{eqnarray}
\Ric_{g} |_{TM} & = & \Ric_{g_M} -n \frac{\mathrm{Hess}_g f}{f} ,\\
\Ric_g |_{T\sphere^n} & = & \left( \frac{(n-1)(1-|\nabla_{g_M} f|_g^2)}{f^2} - \frac{\Delta_g f }{f} \right) f^2 g_{\sphere^n } .
\end{eqnarray}

As a corollary, one can conclude that $\Ric_{ g_M + f^2 g_{\sphere^n} } \geq \lambda $, if and only if for any $c\in (0,1]$, $\Ric_{ g_M + c^2 f^2 g_{\sphere^n} } \geq \lambda $. See also \cite[Remark 3.2]{hnw1}.

\subsection{Adding conical singularities}
Now we recall the lemma Hupp-Naber-Wang used in \cite{hnw1} to add conical singularities to the given manifold.

\begin{lmm}[{\cite[Lemma 5.1]{hnw1}}]
\label{huppnaberwanglemma}
Let $( M_B ,g_B )$ be an $n+1$-dimensional Riemannian manifold, $n\geq 2$, $p\in M_B$, and $f$ be a positive smooth function on $B_2^{n+1} (p) $. Assume that $B_3^{n+1} (p) $ is pre-compact in $( M_B ,g_B )$, and the warped product space $(B_2^{n+1} (p) \times \sphere^2,g)$ with metric $g = g_B + f^2 g_{\sphere^2}$ satisfying 
\begin{align} 
\Ric_g >\lambda \, ,\inj_{g_B}(p) &\geq 2\, , \label{huppnaberwangcondition1}\\
	|\mathrm{Rm}_{g_B}|_{g_B}\, ,\, |\nabla\mathrm{Rm}_{g_B}|_{g_B}\, ,\, |\nabla^2\mathrm{Rm}_{g_B}|_{g_B}&\leq \frac{1}{2} \, , \label{huppnaberwangcondition2}\\
	|\nabla_{g_B}\ln f|_{g_B}\, ,\, |\nabla^2_{g_B}\ln f|_{g_B}\, ,\, |\nabla^3_{g_B} \ln f|_{g_B} &\leq \frac{1}{2} \, .\label{huppnaberwangcondition3}
\end{align}
Write $r := \dist_{g_B}(\cdot,p)$.
Then for all choices of parameters $0 < \eps < \eps_0(n,|\lambda|)$, $0<\alpha < \alpha_0(n,\eps)$, $0 < \hat{r} < \hat r_0(n,\alpha, \lambda,\eps)$, and $0 < \hat{\delta} < \hat{\delta}_0(n,\lambda,\alpha,\epsilon,||f||_{L^\infty},\hat{r})$, there exists $0 < \delta = \Psi(\,\hat{\delta}\|f\|_\infty\mid n, \alpha, \lambda, \eps)$ and a warped product metric $\hat g = \hat{g}_B+\hat f^{\,2} g_{\sphere^2}$ such that:
\begin{enumerate}[\rm (1')]
	\item The Ricci lower bound $\Ric_{\hat g}>\lambda - C(n)\eps$ holds for $\hat r/2 \leq r \leq 2$\, ,
	\item $\hat g = g_B+\hat{\delta}^2 f^2 g_{\sphere^2}$, $ \forall r\in [1,2] $ ,
	\item $\hat g = dr^2+(1-\eps)^2 r^2 g_{\sphere^n}+\delta^2 r^{2\alpha} g_{\sphere^2}$, $\forall r\leq \hat r$ ,
	\item The identity map $\Id:(B_2(p),\hat{g}_B)\to (B_2(p),g_B)$ is $(1+2\epsilon)$-bi-Lipschitz.
\end{enumerate}
\end{lmm}

\begin{rmk}
In this paper, $\Psi \left( \epsilon_1 ,\cdots ,\epsilon_m \big| u_1 ,\cdots ,u_n \right)$ denotes a function of $\epsilon_1 ,\cdots ,\epsilon_m $, $u_1 ,\cdots ,u_n $, such that for fixed $u_i$ we have $ \lim_{\epsilon_1 ,\cdots ,\epsilon_m \to 0} \Psi =0 $.
\end{rmk}

This is a slight generalization of Hupp-Naber-Wang's lemma. Although they chose $n=4$ in their original statement, the proof is essentially verbatim.

\section{Construction of the local model spaces}
\label{sectionlocalmodel}
In this section, we construct some local model spaces by double warped products.

At first, we construct a metric $g_{\varphi ,\phi}$ on $ \mathbb{R}^3 \times \sphere^n $ such that the asymptotic cone is $C(\sphere^n_{1-\epsilon })$.

\begin{lmm}\label{lmm3dimlocalmodel}
Let $ m = 2 $ and $ n\geq 2 $. Then for any $0< \epsilon \leq \frac{1}{100} $, $0<\alpha \leq \alpha_0 (n,\epsilon )<1 $ and $0<\delta \leq \delta_0 (n,\epsilon ,\alpha) $, we can find positive constants $R(n,\epsilon ,\alpha ) $, $C(n,\epsilon ,\alpha ) $ and functions $ \varphi ,\phi $ satisfying (\ref{conditionphirho}) and (\ref{conditionvarphi}), such that $\Ric_{g_{\varphi ,\phi}} > 0$, $\phi (0) =1$, $ \phi |_{[R,\infty )} = (1-\epsilon ) (r+C) $, $\varphi (0)=0$, and $ \varphi |_{[R,\infty )} = \delta (r+C)^{\alpha} $.

Moreover, there exists a small constant $\tau = \tau (n,\epsilon ,\alpha) >0 $ such that $\Ric_{g_{\varphi ,\phi}} \geq \frac{\tau }{r^2 +1} $.
\end{lmm}

\begin{rmk}
Note that $R,C$ and $\tau $ are independent of $\delta $.
\end{rmk}

\begin{proof}
Our construction is divided into three parts.

\smallskip

\par {\em Part 1.}
Let $\varphi_0 (r) = \sin ( r) $ and $\phi_0 (r) = 1+r^4 $ on $\left[ 0, \frac{1}{9n} \right] $. Then Lemma \ref{lmmriccicurvature} implies that
\begin{eqnarray*}
\Ric_{g_{\varphi_0 , \phi_0 } } \left(X_0 , X_0 \right) & = & 2 - n \frac{12 r^2}{1+r^4} \geq 2 - \frac{12}{81} >0 ,\\
\Ric_{g_{\varphi_0 , \phi_0 } } \left(X_1 ,X_1 \right) & = & \varphi_0^2 \cdot \left[2 - n \frac{4 \cos ( r) r^3}{(1+r^4) \sin ( r) } \right] \geq \varphi_0^2 \cdot \left(2 - 4n r^2 \right) >0 ,\\
\Ric_{g_{\varphi_0 , \phi_0 } } \left(X_2 ,X_2 \right) & = & \phi_0^2 \cdot \left[-  \frac{12 r^2}{1+r^4} +(n-1) \frac{1- 16r^6 }{(1+r^4)^2} - 2 \frac{4 \cos ( r) r^3}{(1+r^4) \sin ( r) } \right] \\
& \geq & \phi_0^2 \cdot \left(-  \frac{12}{81} + \frac{1 }{2} - \frac{8}{81 } \right) >0 .
\end{eqnarray*}
Note that $\tan r\geq r$, $\forall r\in (0,1]$. Clearly, we have $\phi_0'(\frac{1}{10n} ) = \frac{1}{250n^3} $ and $\varphi_0'(\frac{1}{10n} ) = \cos (\frac{1}{10n} ) \geq \frac{1}{2} $. 

Write $c_1 =  (10n)^{-3} $ and $R_1 = 100 c_1^{-1} $. Then we can construct $\varphi_1 ,\phi_1$ on $[ 0, \infty ) $ such that
\begin{equation}
 \left\{
\begin{aligned}
\varphi_1 (r) = \varphi_0 (r) &,\;\;\;\; r\in \left[ 0, \frac{1}{10n} \right] ,\\
\varphi_1'' (r) < \varphi_0'' (r) \;\; &,\;\;\;\; r\in \left[ \frac{1}{10n} , \frac{1}{9n} \right] ,\\
\varphi_1' (\frac{1}{9n}) = 2 c_1 \;\; &,\;\;\;\; \\
\varphi_1'' (r) < 0 \;\; &,\;\;\;\; r\in \left[ \frac{1}{9n}, R_1 \right] ,\\
\varphi_1' (R_1 ) = c_1 \;\; &, \\
\end{aligned}
\right.
\end{equation}
and
\begin{equation}
 \left\{
\begin{aligned}
\phi_1 (r) = \phi_0 (r) &,\;\;\;\; r\in \left[ 0, \frac{1}{10n} \right] ,\\
\phi_1'' (r) \leq \phi''_0 (r) &,\;\;\;\; r\in \left[ \frac{1}{10n} , \frac{1}{9n} \right] ,\\
\phi_1'' (r) < 0 \;\; &,\;\;\;\; r\in \left[ \frac{1}{9n}, R_1 \right] ,\\
\phi_1' (R_1 ) = c_1 \;\; & .\\
\end{aligned}
\right.
\end{equation}

Then for $\frac{1}{10n} \leq r\leq \frac{1}{9n}$, we have
\begin{eqnarray*}
\Ric_{g_{\varphi_1 , \phi_1 } } \left(X_0 , X_0 \right) & \geq &   \frac{2\sin (r)}{r} - 12n r^2 >0 ,\\
\Ric_{g_{\varphi_1 , \phi_1 } } \left(X_1 ,X_1 \right) & \geq & \varphi_1 \cdot \left[\sin (r) - 4 n r^3 \right] >0 ,\\
\Ric_{g_{\varphi_1 , \phi_1 } } \left(X_2 ,X_2 \right) & \geq & \phi_1 \cdot \left[-  12 r^2 +(n-1) \frac{1- 16r^6 }{ 1+r^4 } - \frac{8 r^3}{ \sin ( \frac{1}{10n} ) } \right] \\
& \geq & \phi_1 \cdot \left(-  \frac{12}{81n^2} + \frac{1 }{2} - \frac{81n}{729n^3 } \right) >0 .
\end{eqnarray*}

Similarly, when $\frac{1}{9n} \leq r\leq R_1$, we have 
\begin{eqnarray*}
\Ric_{g_{\varphi_1 , \phi_1 } } \left(X_0 , X_0 \right) & > &  0 ,\\
\Ric_{g_{\varphi_1 , \phi_1 } } \left(X_1 ,X_1 \right) & = & \varphi_1 \cdot \left[- \varphi_1'' + \frac{1-(\varphi_1')^2 }{\varphi_1 } - n \frac{\varphi_1' \phi_1'}{ \phi_1 } \right] \\
& > & \varphi_1 \cdot \left[ \frac{1-(2c_1)^2 }{\varphi_1 } - 2n c_1 \right] >0 ,\\
\Ric_{g_{\varphi_1 , \phi_1 } } \left(X_2 ,X_2 \right) & = & \phi_1 \cdot \left[- \phi_1'' + (n-1) \frac{1-(\phi_1')^2 }{\phi_1 } - 2 \frac{\phi_1' \varphi_1'}{ \varphi_1 } \right] \\
& > & \phi_1 \cdot \left(  \frac{1 }{2} - \frac{16c_1}{(9n)^3 \sin (\frac{1}{10n}) } \right) >0 .
\end{eqnarray*}

For convenience, we can assume that there exists a constant $\xi \in (0,R_1 ) $ such that 
$$ \varphi_1 (R_1 ) = \phi_1 (R_1 ) = c_1 (R_1 +\xi ) .$$

\smallskip

\par {\em Part 2.} 
For any $ \alpha \leq \frac{c_1 \epsilon}{10 n } $, let $c_2 = e^{-R_1^2 \alpha^{-2} c_1^{-2} \epsilon^{-2} }$ and $\delta_0 = c_1 (R_1 +\xi ) (R_1 +c_2^{-1} )^{-\alpha } $.

Define $ \varphi_2 (r ) = \delta_0 (r+c_2^{-1})^\alpha $ on $[R_1 ,\infty )$. Then we have 
$$ \varphi'_{2 } (R_1 ) = \alpha (R_1+c_2^{-1})^{-1} \varphi_{2 } (R_1 ) = \alpha c_1 (R_1 +\xi ) (R_1+c_2^{-1})^{-1} \leq c_1 =  \varphi_1' (R_1 ).$$

Now we consider the Ricci curvature of $\Ric_{g_{\varphi_2 , \phi } }$. By a straightforward calculation, we have
\begin{eqnarray}
\Ric_{g_{\varphi_2 , \phi } } \left(X_0 ,X_0 \right) & = & - \left( 2 \frac{\varphi_2''}{\varphi_2} + n \frac{\phi''}{\phi} \right) \label{formulatau11}\\
& = & 2\alpha (1-\alpha ) (r+c_2^{-1})^{-2} - n \frac{\phi''}{\phi} \geq   \alpha (r+c_2^{-1})^{-2} - n \frac{\phi''}{\phi} , \notag \\
\Ric_{g_{\varphi_2 , \phi } } \left(X_1 ,X_1 \right) & = & \varphi_2^2 \cdot \left[-  \frac{\varphi_2''}{\varphi_2} + \frac{1-(\varphi_2')^2 }{\varphi_2^2} - n \frac{\varphi_2' \phi'}{\varphi_2 \phi } \right] \label{formulatau12}\\
& \geq & \varphi_2^2 \cdot \left[ \alpha (1-\alpha ) (r+c_2^{-1})^{-2} + \frac{1 }{2(r+c_2^{-1})^{2 \alpha}} - \alpha n \frac{ \phi'}{(r+c_2^{-1}) \phi } \right] , \notag\\
\Ric_{g_{\varphi_2 , \phi } } \left(X_2 ,X_2 \right) & = & \phi^2 \cdot \left[-  \frac{\phi''}{\phi} +(n-1) \frac{1-(\phi')^2 }{\phi^2} - 2 \frac{\varphi_2' \phi'}{\varphi_2 \phi } \right] \label{formulatau13}\\
& = & \phi\cdot \left[-  \phi'' +(n-1) \frac{1-(\phi')^2 }{\phi } - 2\alpha \frac{ \phi'}{ (r+c_2^{-1}) } \right] . \notag
\end{eqnarray}

Now we can find a positive function $\phi_2 $ on $[R_1 ,\infty )$ such that
\begin{equation}
 \left\{
\begin{aligned}
& \phi_2 (R_1) = \phi_1 (R_1), \;\;\;\; \\
& \phi_2' (R_1) = \phi_1' (R_1) =c_1, \;\;\;\; \\
& 0< \phi_2'' (r ) < \frac{ c_2^3 }{r}, \quad\quad r\in \left[ R_1 , \infty \right) ,\\
& \phi'_2 (r) = 1-\frac{\epsilon}{2},  \quad\quad\;\, r\geq e^{10c_2^{-10}} .
\end{aligned}
\right.
\end{equation}
Hence we have $ 1-(\phi'_2)^2 \geq \frac{\epsilon}{2} $, $c_1 \leq \phi'_2 \leq 1-\frac{\epsilon}{2} $, and $c_1 r\leq\phi_2 \leq r$. It follows that when $r\geq R_1$, $\Ric_{g_{\varphi_2 , \phi_2 } } >0 $. Then we can find a constant $R>R_1$ such that $\phi_2 (R) = (1-\epsilon ) R + (1-\epsilon ) c_2^{-1} $. Since $\phi''_2 >0$, we have $\phi'_2 (R) > 1-\epsilon $.

Let $ \varphi_3 $, $ \phi_3 $ be non-negative functions on $[0,\infty )$ defined by
\begin{equation}
 \varphi_3 (r) = \left\{
\begin{aligned}
\varphi_1 (r) \;\;\;\; &,\;\;\;\; r\leq R_1 ,\\
\delta_0 (r+c_2^{-1})^\alpha &,\;\;\;\; r\geq R_1 ,\\
\end{aligned}
\right.
\end{equation}
and
\begin{equation}
 \phi_3 (r) = \left\{
\begin{aligned}
\phi_1 (r) \;\;\;\; &,\;\;\;\; r\leq R_1 ,\\
\phi_2 (r) \;\;\;\; &,\;\;\;\; R_1 \leq r\leq R ,\\
(1-\epsilon ) (r+c_2^{-1} ) &, \;\;\;\; r\geq R . \\
\end{aligned}
\right.
\end{equation}

By Lemma \ref{lmmgluing}, there are functions $\tilde{\varphi }$ and $\tilde{\phi }$ such that $\Ric_{g_{\tilde{\varphi } , \tilde{\phi } } } >0 $, and for any $r\in \left[ 0,\frac{R_1}{2} \right] \cup [2R ,\infty )$, we have $\tilde{\varphi } = \varphi_3 $ and $\tilde{\phi } = \phi_3 $. Consider the case $ \phi = (1-\epsilon ) (r+c_2^{-1} ) $ in (\ref{formulatau11})-(\ref{formulatau13}), we see that there exists a constant $\tau_1 = \tau_1 (n,\epsilon ,\alpha) >0 $ such that $\Ric_{g_{\tilde{\varphi } , \tilde{\phi } } } >\frac{\tau_1}{1+r^2} $, $ \forall r \geq 2R $. Then $\Ric_{g_{\tilde{\varphi } , \tilde{\phi } } } >0 $ implies that there exists a constant $\tau_2 = \tau_2 (n,\epsilon ,\alpha) >0 $ such that $\Ric_{g_{\tilde{\varphi } , \tilde{\phi } } } \geq \frac{\tau_2}{1+r^2} $, $ \forall r \geq 0 $.

\par {\em Part 3.} 
Now we consider the case $\delta < \delta_0 $. Write $s=\delta_0^{-1} \delta 
\in (0,1) $. For any $0 < t\leq (\delta c_2)^5 $, we define
\begin{equation}
 \tilde{\varphi }_t (r) = \left\{
\begin{aligned}
t \sin (t^{-1} r) \;\;\;  &,\;\;\;\; r\leq t \arccos (s) ,\\
s \tilde{\varphi } \left( r + \theta (t,s) \right) &,\;\;\;\; r\geq t \arccos (s) ,\\
\end{aligned}
\right.
\end{equation}
where $ \theta (t,s) = \arcsin \left( t \sqrt{ s^{-2}-1} \right) - t \arccos (s) \in (0,\frac{1}{100n^2})$.

It is straightforward to check that for any $r\leq t \arccos (s)$, $\Ric_{g_{\tilde{\varphi }_t , \tilde{\phi } } } \geq \frac{1}{2t} $. When $t\arccos ( s ) \leq r\leq \frac{1}{20n} $, we have $\tilde{\varphi }_t (r)=s \varphi_0 (r+\theta(t,s))$ and $\tilde{\phi}(r)=\phi_0(r)$. Thus Lemma \ref{lmmriccicurvature} applies and yields
\begin{eqnarray*}
\Ric_{g_{\tilde{\varphi }_t ,  \tilde{\phi } } } \left(X_0 , X_0 \right) & \geq & 2 - n \frac{12 r^2}{1+r^4} \geq 2 - \frac{12}{10} \geq \frac{1}{2} ,\\
\Ric_{g_{\tilde{\varphi }_t ,  \tilde{\phi } } } \left(X_1 ,X_1 \right) & \geq & \tilde{\varphi }_t^2 \cdot \left(  1 - n \frac{4r^3 \cos \left( r + \theta (t,s) \right) }{ (1+r^4)\sin \left( r + \theta (t,s) \right) } \right) \geq \frac{ \tilde{\varphi }_t^2 }{2} ,\\
\Ric_{g_{\tilde{\varphi }_t ,  \tilde{\phi } } } \left(X_2 ,X_2 \right) & \geq & \tilde{\phi } \cdot \left[-  12r^2 +(n-1) \frac{1- 16r^6 }{2} - 2\frac{4r^3 \cos \left( r + \theta (t,s) \right) }{ \sin \left( r + \theta (t,s) \right) } \right] \\
& \geq & \tilde{\phi } \cdot \left(-  \frac{12}{81} + \frac{1 }{3} - \frac{8}{81 } \right) \geq \frac{\tilde{\phi } }{81} .
\end{eqnarray*}
By Lemma \ref{lmmgluing}, for any given $s,t$, we can modify $\tilde{\varphi }$ and $\tilde{\phi }$ near $t\arccos (s) $ such that they became smooth functions, and $\Ric_{g_{\tilde{\varphi }_t , \tilde{\phi } } } \geq \frac{1}{100} $ for $r\leq \frac{1}{20n} $.

Choosing a cut-off function $\eta\in C^\infty ([0,\infty ))$ such that $0\leq \eta \leq 1$, $\eta (r) =1$ when $r\leq \frac{1}{100n}$, and $\eta (r) =0 $ when $r\geq \frac{1}{50n}$. Now we can see that there exists a constant $t_\delta >0$ such that for any $t\in (0,2 t_\delta )$, we have $\Ric_{g_{\eta \tilde{\varphi }_t + (1-\eta ) \tilde{\varphi }, \tilde{\phi } } }  \geq \frac{\tau_2}{2(1+r^2 )} $, $ \forall r \geq 0 $. Note that when $r\geq \frac{1}{200n}$, $ \tilde{\varphi }_t $ converges to $ s\tilde{\varphi } $ as $t\to 0$ in the $C^\infty $-sense. Therefore, the lemma follows by setting $\varphi = \eta \tilde{\varphi }_{ t_\delta} + (1-\eta ) \tilde{\varphi } $, $\phi = \tilde{\phi }$, $C = c_2^{-1}$, $R = c_2^{-10} $, and $\tau = c_2 \tau_2 $.
\end{proof}

As a corollary, we have:

\begin{coro}\label{coro3dimlocalmodel}
Let $ n\geq 2 $. Then for any $ 0< \epsilon \leq \frac{1}{100} $, $0<\alpha \leq \alpha_0 (n,\epsilon ) $, $0<\kappa \leq \frac{1}{100} $, and $0<\delta \leq \delta_0 (n,\epsilon ,\alpha ,\kappa) $, we can find a constant $ \tau = \tau (n,\epsilon ,\alpha) \in (0,1) $, a Riemannian manifold $(\mathcal{B} , g ) = (\mathcal{B} (\epsilon ,\alpha ,\delta ,\kappa) , g_{\epsilon ,\alpha ,\delta ,\kappa })$, a compact totally geodesic submanifold $\mathcal{Z} \subset \mathcal{B}$ isometric to $\sphere_{\mu}^n $, and a diffeomorphism $ F: B_1^3 (0) \times \sphere^{n} \to  \mathcal{B} $ such that:
\begin{enumerate}[\rm (1'')]
    \item $\Ric_{g } >\frac{\tau}{r^2+\kappa } $ on $\mathcal{B} $,
    \item $F^{-1} (\mathcal{Z}) = \{ 0 \} \times \sphere^{n} $,
    \item The pull-back $F^* g$ is given by $g_{\varphi ,\phi }$ for some nonnegative functions satisfying (\ref{conditionphirho}) and (\ref{conditionvarphi}),
    \item There exists an open subset $\mathcal{U} \subset \mathcal{B} $ such that $\mathcal{B} \sq \mathcal{U} $ is isometric to 
$$\left( (\kappa ,1) \times \sphere^2 \times \sphere^{n } , dr^2 + \delta^2 r^{2\alpha} g_{ \sphere^2 } + (1-\epsilon )^2 r^2 g_{\sphere^{n } } \right) . $$
\end{enumerate}
\end{coro}

\begin{proof}
By Lemma \ref{lmm3dimlocalmodel}, for any positive $ \epsilon \leq \frac{1}{100} $, $\alpha \leq \alpha_0 (n,\epsilon ) $ and $\delta_1 \leq \delta_0 (n,\epsilon ,\alpha) $, we can find positive constants $R(n,\epsilon ,\alpha ) $, $C(n,\epsilon ,\alpha ) $, $\tau = \tau (n,\epsilon ,\alpha) >0 $ and functions $ \varphi_1 ,\phi_1 $ satisfying (\ref{conditionphirho}) and (\ref{conditionvarphi}), such that $\Ric_{g_{\varphi_1 ,\phi_1 }} \geq \frac{\tau }{r^2 +1}$, $\phi_1 (0) =1$, $ \phi_1 |_{[R,\infty )} = (1-\epsilon )r $, $\varphi_1 (0)=0$, and $ \varphi_1 |_{[R,\infty )} = \delta_1 (r+C)^{\frac{\alpha}{2}} $.

Let $t \geq \frac{2R}{\kappa} $, $ \varphi (r) = t^{-1} \varphi_1 (tr) $ and $ \phi (r) = t^{-1} \phi_1 (tr) $. Then we have $ \varphi (r) =  t^{-1+\alpha } r^\alpha $ and $ \phi (r) = (1-\epsilon )r $, $r\in [\kappa ,\infty )$. Clearly, the rescaled metric $t^{-2} g_{\varphi_1 ,\phi_1 } $ can be represented by
$$\left( (0 ,1) \times \sphere^2 \times \sphere^{n} , dr^2 +  \varphi ( r)^2 g_{ \sphere^2 } + \phi ( r)^2 g_{\sphere^{n} } \right) . $$
It follows that 
$$\Ric_{g_{\varphi ,\phi} } |_{r=r_0} = t^{ 2}\Ric_{g_{\varphi_1 ,\phi_1} } |_{r=tr_0} \geq \frac{t^2\tau}{t^2 r_0^2 +1 } = \frac{ \tau}{ r_0^2 +t^{-2} } \geq \frac{\tau}{r_0^2 +\kappa} ,\;\; \forall r_0 >0 .$$
Since $\varphi (0)=0$, $\varphi'(0 ) =1 $ and $\phi (0) = t^{-1} $, we see that $g_{\varphi ,\phi}$ gives a Riemannian metric on $ B_1^3 (0) \times \sphere^{n} $, and the subset $\mathcal{Z}$ corresponding to $r=0$ is isometric to $\sphere_{t^{-1}}^n $. This completes the proof.
\end{proof}

By the last part of the proof of Lemma \ref{lmm3dimlocalmodel}, one can obtain the following result.

\begin{coro}\label{corollarymodify2}
Let $ m = 2 $, $ n\geq 2 $, let $\lambda ,\mu >0$, and let $ \varphi ,\phi $ be functions on $[0,\mu ]$ satisfying (\ref{conditionphirho}), (\ref{conditionvarphi}), and $\Ric_{g_{\varphi ,\phi}} > \lambda $. Assume that there exists an open neighborhood $U$ of $0\in [0,\mu ]$ such that
$$ \varphi (r) = A^{-1} \sin (Ar),\quad \phi (r) = A^{-1} + A^3 r^4 ,\quad\textrm{ on $U$,} $$
for some constant $A>0$. Then for any $\delta ,\epsilon \in (0,1) $, there exists a function $ \tilde{\varphi} \geq 0 $ satisfying (\ref{conditionvarphi}) , $\Ric_{g_{\varphi ,\tilde{\phi}}} > \lambda $, $|\tilde{\varphi}-\delta\varphi |\leq \epsilon $, and $ \tilde{\varphi} =\delta \varphi $ on $\left[ \frac{\mu}{2} , \mu \right] $.
\end{coro}

\section{Construction of \texorpdfstring{$(M_{i } ,g_{i } )$}{Lg}}
\label{holessection}
In this section, we will construct the sequence of $(n+2)$-dimensional manifolds $(M_{i} ,g_i )  \stackrel{GH}{\longrightarrow} (X_\epsilon ,d_\epsilon ) $ for some given $\epsilon$. Our construction is based on an induction process.

\subsection{Puncture lemma} At first, in order to make the first hole on $X$, we need following lemma. This lemma is a combination of Lemma \ref{lmm3dimlocalmodel} and Lemma \ref{huppnaberwanglemma}. Since we plan to use this lemma to the interior of a manifold with boundary, we need to consider non-compact manifolds here.

\begin{lmm} \label{lmmpuncture}
Let $( M_B ,g_B )$ be an $n $-dimensional Riemannian manifold (need not be complete), $n\geq 3$, $p\in M_B$, $\mathcal{U}_p$ be a pre-compact open neighborhood of $p $, and $f$ be a positive smooth function on $\bar{\mathcal{U}}_p$. Assume that the warped product space $(\mathcal{U}_p \times \sphere^2,g)$ with metric $g = g_B + f^2 g_{\sphere^2}$ satisfying that $ \Ric_g >\lambda \in\mathbb{R} $. Then there exist constants $r_0\in (0,1 ]$, $ \epsilon_0 \in (0,\frac{1}{100} )$ depending on $(\mathcal{U}_p \times \sphere^2,g)$, such that for any $0< \epsilon < \epsilon_0 $, $0< r^* \leq r_0 ( g ,\epsilon ) $, and $0< \delta \leq \delta_0 (g,\epsilon ,r^* ) $, there exists an $(n+2) $-dimensional Riemannian manifold $(\mathcal{V}_p ,\hat{ {g}})$, satisfying the following properties:
\begin{enumerate}[(i).]
    \item The closed ball $\bar{B}^{g_B}_{4r_0} (p) \subset ( M_B ,g_B ) $ is a compact subset of $\mathcal{U}_p$,
	\item $\Ric_{\hat g}> \lambda $ on $\mathcal{V}_p $.
\end{enumerate}

Moreover, there exist an $n $-dimensional compact totally geodesic submanifold $\mathcal{Z} \subset \mathcal{V}_p $ isometric to $\sphere_{r^*}^n $, and a diffeomorphism $F: \mathcal{V}_p \sq \mathcal{Z} \to ( \mathcal{U}_p \sq \bar{B}^{g_B}_{r^*} (p)  ) \times \sphere^2 $, such that:
\begin{enumerate}[(i).]
\addtocounter{enumi}{2}
    \item The pull-back of metric $(F^{-1} )^* \hat{g} $ is given by the warped product $ \left( ( \mathcal{U}_p \sq \bar{B}^{g_B}_{r^*} (p)  ) \times \sphere^2 ,\hat{ {g}}_B + \hat{f}^2 g_{\sphere^2}\right) $,
    \item $\| \hat{f} \|_{L^\infty (B^{g_B}_{3r_0} (p)) } \leq \delta \| f \|_{L^\infty (B^{g_B}_{3r_0} (p)) } $,
    \item $\hat{ {g}}_B = g_B $ and $\hat{f} = \delta^2 f $ on $ \mathcal{U}_p \sq B^{g_B}_{2r_0} (p) $.
    \item The Gromov-Hausdorff distance $d_{GH} \left( ( \mathcal{U}_p \sq \bar{B}^{g_B}_{r^*} (p) , d_{\hat{ {g}}_B } ) , ( \mathcal{U}_p , d_{g_B} ) \right) \leq \epsilon $,
    \item The complement $ ( M_B\setminus {B}^{g_B}_{2r^*} (p) , \hat{ {g}}_B ) $ is a smooth Riemannian manifold with boundary,
	\item The continuous extension of $\pi_{\mathcal{U}_p} \circ F :(\mathcal{V}_p \sq \mathcal{Z} ,\hat{g} )\to \left( \mathcal{U}_p \sq \bar{B}^{g_B}_{r^*} (p) , \hat{ {g}}_B \right) $ gives an isometry $\mathcal{Z} \to \partial{B}^{\hat{ {g}}_B}_{r^*} (p) $,
    \item The restriction of inclusion map $\Id :\left( \mathcal{U}_p \sq \bar{B}_{r^* } (p) , d_{\hat{ {g}}_B } \right) \to \left( \mathcal{U}_p , d_{g_B } \right) $ on $\mathcal{U}_p \sq \bar{B}^{g_B}_{r_0 } (p) $ is $(1+\epsilon)$-bi-Lipschitz.
\end{enumerate}
\end{lmm}

\begin{rmk}
As we pointed out in Corollary \ref{coro3dimlocalmodel}, if $(\mathcal{V}_p ,\hat{g})$ is the Riemannian manifold we constructed in this lemma, then for any $t \in (0,1)$ and any open neighborhood $\mathcal{U}_{\mathcal{Z}}$ of $\mathcal{Z}$, there exists another metric $\hat{g}_t $ on $ \mathcal{V}_p $ such that $\Ric_{\hat{g}_t } >\tau $, and the pull-back $(F^{-1} )^* \hat{g}_t $ is given by the warped product $\hat{g}_B + \hat{f}_t^2 $, and $\hat{f}_t = t\hat{f} $ outside $F^{-1} ( \mathcal{U}_{\mathcal{Z}} ) $.
\end{rmk}

\begin{proof}
Our construction is divided into three parts. The first step is to add a conical singularity to the manifold $M_B$ using Lemma \ref{huppnaberwanglemma}. The second step is to change the conical singularity we added to a hole using Lemma \ref{lmm3dimlocalmodel} or Corollary \ref{coro3dimlocalmodel}. In the first two parts, we will prove (i) and (ii). In the last part, we will construct the map $F$ and prove the properties (iii)-(ix).

\smallskip

\par {\em Part 1.}
By scaling, one can find a small constant $r_B = r_B (g_B) >0$ such that $\bar{B}^{g}_{10 r_B} (p) $ is a compact subset of $\mathcal{U}_p$, and the manifold $ \left( B^{r^{-2}_B g_B}_{2 } (p) \times \sphere^2 , r^{-2}_B g \right) = \left( B^{g}_{2r_B} (p) \times \sphere^2 , r_B^{-2} g_B + r_B^{-2} f^2 g_{\sphere^2} \right) $ satisfying the conditions (\ref{huppnaberwangcondition1})-(\ref{huppnaberwangcondition3}). Note that for any $\lambda \in (0,\infty )$, the data $(g_B ,f)$ satisfies (\ref{huppnaberwangcondition3}) if and only if $(g_B ,\lambda f)$ satisfies (\ref{huppnaberwangcondition3}). By replacing $g$ by $r^{-2}_B g $, we can assume that $\bar{B}^{g_B}_{10 } (p) $ is a compact subset of $\mathcal{U}_p$, and the manifold $ \left( B^{ g_B}_{2 } (p) \times \sphere^2 , g \right) $ satisfying the conditions in Lemma \ref{huppnaberwanglemma}. Then one can obtain (i) by choosing $r_0=1$. Without loss of generality, we can assume that $\Ric_{ g} \geq \lambda + \mu $ for some constant $\mu >0 $ on $\bar{B}^{g_B}_{10 } (p) \times \sphere^2 $.

Let $\hat g_1 = \hat{g}_{B,1}+\hat f_1^{\,2} g_{\sphere^2}$ be the warped product metric on $ B_{2} (p) \times \sphere^2 $ corresponding to the data $(\epsilon_1 , \alpha_1 , \hat r_1 , \hat{\delta}_1)$ constructed in Lemma \ref{huppnaberwanglemma}. 

Now we will describe how to select the data $(\epsilon_1 , \alpha_1 , \hat r_1 , \hat{\delta}_1)$. By Lemma \ref{huppnaberwanglemma}, there exists a constant $C=C(n) >0$ such that the Ricci lower bound $\Ric_{\hat g_1 }> (\lambda + \mu ) - C(n)\eps_1 $ holds for $\frac{\hat r_1}{2} \leq r \leq 2$. Hence we can conclude that $\Ric_{\hat g_1 }\geq \lambda + \frac{\mu}{2} $ holds for $\frac{ \hat r_1}{2} \leq r \leq 2 $ if $\epsilon_1 \leq \frac{ \mu}{10C} $. Then we choose $\epsilon_1 \leq \frac{ \mu}{20C} $, and $\alpha_1 \leq \alpha (n ,|\lambda| ) $ such that the identity map $\Id:(B_{2 }(p),\hat{g}_{B,1})\to (B_{2 }(p), g_B)$ is $(1+ \epsilon^2)$-bi-Lipschitz.

Now we can apply Lemma \ref{huppnaberwanglemma} again to show that $\hat g_1 = dr^2+(1-\eps_1 )^2 r^2 g_{\sphere^n}+\delta_1^2 r^{2\alpha} g_{\sphere^2}$, $\forall r\leq \hat r_1$, where $\delta_1 = \Psi \left( \hat{\delta}_1 \|f_0 \|_{L^\infty} \bigg| n, \alpha_1, \lambda, \eps_1 \right) $. Now we move on to the next step, which is to make a hole at $p$. We will determine the constants $\hat r_1 $ and $\hat \delta_1 $ in the last part.

\smallskip

\par {\em Part 2.}
By Corollary \ref{coro3dimlocalmodel}, for any constants $\kappa_1 \leq \kappa_0 (n,\epsilon_1 ,\alpha_1 ) <1 $ and $ \delta'_1 \leq \delta'_0 (n,\epsilon_1 ,\alpha_1 ,\kappa_1 ) <1 $, then there exist a constant $\tau (n,\epsilon_1 ,\alpha_1 ) >0$, a Riemannian manifold $(\mathcal{B} , g_\mathcal{B} ) $, a compact submanifold $\mathcal{Z} \subset \mathcal{B}$ isometric to $\sphere_{\sigma}^{n-1} $, and a diffeomorphism $ F_\mathcal{B} : B_1^3 (0) \times \sphere^{n-1} \to  \mathcal{B} $ such that:
\begin{itemize}
    \item $0< \sigma = \Psi (\kappa_1 | n,\epsilon_1 ,\alpha_1 ) $,
    \item $\Ric_{g_\mathcal{B} } >\frac{\tau}{r^2+\kappa_1 } $ on $\mathcal{B} $,
    \item $F_\mathcal{B}^{-1} (\mathcal{Z}) = \{ 0 \} \times \sphere^{n-1} $,
    \item The pull-back $F_\mathcal{B}^* g_\mathcal{B} $ is given by $g_{\varphi ,\phi }$ for some nonnegative functions satisfying (\ref{conditionphirho}) and (\ref{conditionvarphi}),
    \item There exists an open subset $\mathcal{U} \subset \mathcal{B} $ such that $\mathcal{B} \sq \mathcal{U} $ is isometric to 
$$\left( (\kappa_1 ,1) \times \sphere^2 \times \sphere^{n-1} , dr^2 + \delta'^2_1 r^{2\alpha_1 } g_{ \sphere^2 } + (1-\epsilon_1 )^2 r^2 g_{\sphere^{n-1} } \right) . $$
\end{itemize}
Let $\hat{r}_1 = 10\kappa_1 $ and $\delta'_1 = \delta_1$. Now we can glue $( F_\mathcal{B} ( B_{2\kappa_1 }^3 (0) \times \sphere^{n-1} ), g_\mathcal{B} ) $ to $\left( B_{2} (p) \times \sphere^2 , \hat{g} \right) $ by the identity map on $ (4\kappa_1 ,5\kappa_1) \times \sphere^2 \times \sphere^{n-1} $. 

Let $(\mathcal{V}_p ,\hat{ {g}})$ denote the manifold of $( F_\mathcal{B} ( B_{2\kappa_1 }^3 (0) \times \sphere^{n-1} ), g_\mathcal{B} ) $ glued to $\left( B_{2} (p) \times \sphere^2 , \hat{g} \right) $. Since $\tau $ is independent of $ \kappa_1 $ and $\delta'_1$, one can see there exists a constant $\kappa_{2}=\kappa_0 (n,\epsilon_1 ,\alpha_1 ,|\lambda | )$ such that when $\kappa_1 \in (0,\kappa_2 ) $, $\Ric_{g_\mathcal{B} } >\frac{\tau}{r+\kappa_1 } >  \lambda $, $\forall r\leq \hat{r}_1 $. Then we can conclude that the manifold $(\mathcal{V}_p ,\hat{ {g}})$ satisfies (ii) when $\hat{r}_1\in (0,\frac{\kappa_2}{10})$.

\smallskip

\par {\em Part 3.}
Let $(B_\mathcal{B} ,g_{B_{\mathcal{B}}} ) = ((0,1) \times \sphere^{n-1} , dr^2 + \phi^2 g_{\sphere^{n-1}} )$. Then we have $\phi = (1-\epsilon_1)r $ for $r\geq \kappa_1$. Then we can glue $(B_\mathcal{B} ,g_{B_{\mathcal{B}}} ) $ to $\left( B_{2} (p) , \hat{g}_{B,1} \right) $ by the identity map on $ (4\kappa_1 ,5\kappa_1) \times \sphere^{n-1} $. Let $(B_{\mathcal{V}_p} ,\hat{ {g}}_{B_{\mathcal{V}_p}})$ denote the manifold of $(B_\mathcal{B} ,g_{B_{\mathcal{B}}} ) $ glued to $\left( B_{2} (p) , \hat{g}_{B,1} \right) $. It is easy to see that there is a natural projection $\pi_{{\mathcal{V}_p}} : {\mathcal{V}_p}\sq \mathcal{Z} \to B_{\mathcal{V}_p} $, and this projection gives a warped product structure $( B_{\mathcal{V}_p} \times \sphere^{2} , \hat{ {g}}_{B_{\mathcal{V}_p}} + \hat{f}_{B_{\mathcal{V}_p}}^2 g_{\sphere^{2}} )$ corresponding to the $\sphere^2$-bundle $\pi_{\mathcal{V}_p} : {\mathcal{V}_p}\sq \mathcal{Z} \to B_{\mathcal{V}_p} $.

Now we construct a Lipschitz map
$$ \varrho : B_\mathcal{B} = \left( (0 ,1) \times \sphere^{n-1} , dr^2 + \phi^2 g_{\sphere^{n-1} } \right) \to \left( (0 ,1) \times \sphere^{n-1} , dr^2 + (1-\epsilon_1 )^2 r^2 g_{\sphere^{n-1} } \right) \subset B_2 (p) $$
by letting $ \varrho (r,x) = ((1-\epsilon_1 )^{-1} \phi (r) ,x) $. It is easy to see that for any given $r'\in (0,1)$, the restriction $\varrho |_{r=r'} $ is a rescaling by a constant. Let $r^* = (1-\epsilon )^{-1} \phi (0)  = \Psi (\kappa_1 | \epsilon_1 , \alpha) $. Then we can get a diffeomorphism $ B_{\mathcal{V}_p} \to \mathcal{U}_p \sq \bar{B}^{g_B}_{r^*} (p)$ by gluing $\varrho$ and the identity map of $\mathcal{U}_p \sq \bar{B}^{g_B}_{3\kappa_1} (p)$. Hence one can construct the map $F$ by combining this diffeomorphism and the warped product structure we constructed above. By the construction of $F$, one can obtain (iii)-(viii).

It is sufficient to prove (ix) now.

Without loss of generality, we assume that $r^* \leq \frac{r_0}{10} $. Let $x,y\in \mathcal{U}_p \sq \bar{B}^{g_B}_{r_0 } (p) $. When $d_{\hat{g}_{B,1}} (x,y) \leq r_0 - 10\kappa_1 $, the intersection of $\bar{B}^{g_B}_{\kappa_1 } (p) $ and the geodesic between $x,y$ in $(\mathcal{U}_p , \hat{g}_{B,1} )$ is empty, and $d_{\hat{g}_{B,1}} (x,y) = d_{\hat{g}_{B }} (x,y)  $. If $d_{\hat{g}_{B,1}} (x,y) \geq r_0 - 10\kappa_1 $, then we have 
$$ d_{\hat{g}_{B }} (x,y) - 10\kappa_1 \leq d_{\hat{g}_{B,1}} (x,y) \leq d_{\hat{g}_{B }} (x,y) +10\kappa_1 .$$
It follows that the restriction of inclusion map $\Id :\left( \mathcal{U}_p \sq \bar{B}_{r^* } (p) , d_{\hat{ {g}}_B } \right) \to \left( \mathcal{U}_p , d_{\hat{g}_{B,1}} \right) $ on $\mathcal{U}_p \sq \bar{B}^{g_B}_{r_0 } (p) $ is $(1+\Psi (\kappa_1 |n,r_0 ))$-bi-Lipschitz. Since the identity map $\Id:(B_{2 }(p),\hat{g}_{B,1})\to (B_{2 }(p), g_B)$ is $(1+ \epsilon^2)$-bi-Lipschitz, we see that (ix) holds when $\kappa_1$ is sufficiently small.
 
Now we describe the choice of $\hat r_1 $ and $\hat \delta_1 $ in the first part. Since $\hat{r}_1 = 10\kappa_1 $, we can determine $\hat r_1 $ by choosing sufficiently small $\kappa_1 >0$. Note that we can always shrink $\hat \delta_1 $ arbitrarily, so we just need to take sufficiently small $\hat \delta_1 $ such that $\delta'_1 =\delta_1 $ is sufficiently small, and all the above arguments hold. This completes the proof.
\end{proof}

\subsection{Induction process}
Now we begin to construct the sequence of $(n+2)$-dimensional manifolds $(M_{i} ,g_i )  \stackrel{GH}{\longrightarrow} (X_\epsilon ,d_\epsilon ) $.

Let $(M_0 ,g_0)$ denotes the complete Riemannian manifold $\left( X\times \sphere^2 , h+ \frac{1}{1+|\lambda|^2} g_{\sphere^2} \right) $. Since $\Ric_h >\lambda $, one can see that $\Ric_{g_0} > \lambda $.

Our next step is to construct $(M_{i } ,g_{i } )$ by induction. We state here the inductive hypothesis we need for induction on $i$.

\begin{inductionhyp}[about $i$]
Fix $p_0 \in X$. Now we assume that we have constructed:
\begin{enumerate}[1)${}_{i}$.]
\item An $(n+2)$-dimensional Riemannian manifold $ (M_i , g_i ) $ with $\Ric_{g_i} > \lambda $,
\item A pre-compact open subset $\mathcal{W}_i = \bigsqcup\limits_{j=1}^{N_i} \mathcal{U}_{j} \subset \bigsqcup\limits_{j=-i}^{i} B^h_{2^{j+1} - 2^{-3i} } (p_0) \sq \bar{B}^h_{2^{j} + 2^{-3i} } (p_0) \subset (X,h) $,
\item An $n $-dimensional union of compact totally geodesic submanifolds $\mathcal{Y}_i = \bigsqcup\limits_{j=1}^{N_i} \mathcal{Z}_{j} \subset M_i $,
\item A Riemannian metric $ g_{X,i} $ and a positive function $f_i$ on $X \sq \mathcal{W}_i $,
\item An isometry $F_i : ( M_i \sq \mathcal{Y}_i ,g_i ) \to \left( (X \sq \bar{\mathcal{W}}_i ) \times \sphere^2 , g_{X,i } + f_i^2 g_{\sphere^2 } \right) $,
\end{enumerate}
such that:
\begin{enumerate}[1)${}_{i}$.]
\addtocounter{enumi}{5}
\item The function $|f_i| \leq 2^{-i} $,
\item The closures $\bar{\mathcal{U}}_{j} $ are diffeomorphic to the closed unit $n$-ball,
\item The completion of $\left( X \sq \bar{\mathcal{W}}_i , d_{g_{X,i}} \right) $ is a smooth Riemannian manifold with boundary $\partial \mathcal{W}_i$,
\item The metric $g_{X,i } = h $ on $ B^h_{2^{-i}} (p_0) \cup \left( X \sq B^h_{2^{i+1}} (p_0) \right) \cup \bigsqcup\limits_{j=-i}^{i} B^h_{2^{j} + 2^{-4i} } (p_0) \sq \bar{B}^h_{2^{j} - 2^{-4i} } (p_0) \subset (X,h) $,
\item The Gromov-Hausdorff distance $d_{GH} \left( ( X \sq \bar{\mathcal{W}}_i )  , d_{g_{X,i }} ) , (X,h) \right) \leq \left( 1-2^{-i} \right) \epsilon $,
\item For any $x\in B^{h}_{2^i} (p_0) \sq \bar{\mathcal{W}}_i \subset \left( X \sq \bar{\mathcal{W}}_i , d_{g_{X,i}} \right) $, the distance $d_{g_{X,i}} (x, \mathcal{W}_i ) \leq 2^{2-i}$,
\item For any $x\in ( B^h_{2^i} \sq \bar{\mathcal{W}}_i , d_{g_{X,i }} )$, there exist a constant $c_i \in (0,2^{-i}) $ and a closed subset $K_{x,i }$ of $\mathbb{R}^n$ satisfies that $ c_i^{-1} d_{GH} \left( B^{g_{X,i}}_{c_i } (x) , K_{x,i } \right) \leq 2^{1-i} (1-2^{-i} ) $. 
\end{enumerate}
\end{inductionhyp}

Note that the last of these property listed in Proposition \ref{propnonmfd} derives from the following fact.

\begin{fact}
Let $(M,g)$ be a smooth Riemannian manifold with boundary, and $\Omega$ is a compact subset of $M$. Then for any $x\in\Omega$ and $\epsilon >0$, there exists a constant $r_0 =r_0 (\Omega ,g) >0$ such that for any $0<r<r_0$, we can find a subset $K_{x,r}$ of $\mathbb{R}^n$ satisfies that
$$ r^{-1} d_{GH} \left( B^{g }_{r } (x) , K_{x,r} \right) \leq \epsilon .$$
\end{fact}

Now we construct the manifold $(M_{i+1} ,g_{i+1} )$.

Let us begin with choosing a discrete subset $P=\{ p_k \}_{k=1}^{N'_i} $ of $\left(\bigsqcup\limits_{j=-i}^{i} B^h_{2^{j+1} - 2^{-3i} } (p_0) \sq \bar{B}^h_{2^{j} + 2^{-3i} } (p_0) \right)\sq \bar{\mathcal{W}}_i $ such that for any $x\in B^{h}_{2^i} (p_0) \sq \mathcal{W}_i \subset \left( X \sq \bar{\mathcal{W}}_i , d_{g_{X,i}} \right) $, the distance $d_{g_{X,i}} (x, \mathcal{W}_i \cup P ) + d_{h} (x, \mathcal{W}_i \cup P ) \leq 2^{-2-i} $.

Then we can construct $(M_{i+1} ,g_{i+1} )$, $\mathcal{W}_{i+1}$, $\mathcal{Y}_{i+1}$, $\left( (X \sq \mathcal{W}_{i+1} ) \times \sphere^2 , g_{X,i+1 } + f_{i+1}^2 g_{\sphere^2 } \right) $ and $F_{i+1}$ by applying Lemma \ref{lmmpuncture} to each point in the discrete set $P=\{ p_k \}_{k=1}^{N'_i} \subset \left( X \sq \bar{\mathcal{W}}_i ,d_{g_{X,i}} \right) $. We will finish it by induction on $k$. We state here the inductive hypothesis we need for induction on $i$.

\begin{inductionhyp}[about $k$]
Fix $p_0 \in X$. Now we assume that we have constructed:
\begin{enumerate}[1)${}_{i,k}$.]
\item An $(n+2)$-dimensional Riemannian manifold $ (M_{i,k} , g_{i,k} ) $ with $\Ric_{g_{i,k}} > \lambda $,
\item A pre-compact open subset 
$$\mathcal{W}_{i,k} = \left( \bigsqcup\limits_{j=1}^{N_i} \mathcal{U}_{j} \right) \bigsqcup \left( \bigsqcup\limits_{l=1}^{k} \mathcal{U}'_{l } \right) \subset \bigsqcup\limits_{j=-i-1}^{i+1} B^h_{2^{j+1} - 2^{-3i-3} } (p_0) \sq \bar{B}^h_{2^{j} + 2^{-3i-3} } (p_0) \subset (X,h) ,$$
\item An $n $-dimensional union of compact totally geodesic submanifolds $\mathcal{Y}_{i,k} = \left( \bigsqcup\limits_{j=1}^{N_i} \mathcal{Z}_{j} \right) \bigsqcup \left( \bigsqcup\limits_{l=1}^{k} \mathcal{Z}'_{l} \right) \subset M_{i,k} $,
\item A Riemannian metric $ g_{X,i,k} $ and a positive function $f_{i,k}$ on $X \sq \mathcal{W}_{i,k} $,
\item An isometry $F_{i,k} : ( M_{i,k} \sq \mathcal{Y}_{i,k} ,g_{i,k} ) \to \left( (X \sq \bar{\mathcal{W}}_{i,k} ) \times \sphere^2 , g_{X,{i,k} } + f_{i,k}^2 g_{\sphere^2 } \right) $,
\end{enumerate}
such that:
\begin{enumerate}[1)${}_{i,k}$.]
\addtocounter{enumi}{5}
\item The closures $\bar{\mathcal{U}}_{j} $ and $\bar{\mathcal{U}}'_{l} $ are diffeomorphic to the closed unit $n$-ball,
\item The completion of $\left( X \sq \bar{\mathcal{W}}_{i,k} , d_{g_{X,i,k}} \right) $ is a smooth Riemannian manifold with boundary $\partial \mathcal{W}_{i,k}$,
\item The metric $g_{X,i,k } = h $ on 
\begin{eqnarray*}
    & & B^h_{2^{-i-1} (1+ 2^{-3i-3-k} ) } (p_0) \cup \left( X \sq B^h_{2^{i+2} (1+ 2^{-5i-3-k} ) } (p_0) \right) \\
    & \bigsqcup & \left(\bigsqcup\limits_{j=-i-1}^{i+1} B^h_{2^{j} + 2^{-4i-4} (1+2^{-3-k}) } (p_0) \sq \bar{B}^h_{2^{j} - 2^{-4i-4} (1+2^{-3-k}) } (p_0) \right) \subset (X,h) ,
\end{eqnarray*}
\item For any $x\in ( B^h_{2^i} \sq \bar{\mathcal{W}}_i , d_{g_{X,i }} )$, there exist a constant $c_i \in (0,2^{-i}) $ and a closed subset $K_{x,i }$ of $\mathbb{R}^n$ satisfies that $ c_i^{-1} d_{GH} \left( B^{g_{X,i}}_{c_i } (x) , K_{x,i } \right) \leq 2^{-i} (2-2^{-i}+k 2^{-i-{N'_i} } )$,
\item For any point $p_l \in P$, $p_l\in \bar{\mathcal{U}}'_{l'}$ if and only if $l=l'$.
\end{enumerate}
\end{inductionhyp}

Now we can apply Lemma \ref{lmmpuncture} around $p_{k+1}$ to get a smooth Riemannian manifold $(\mathcal{V}_{k+1} ,g_{\mathcal{V}} ) $, an open subset $\mathcal{V}'_{k+1}$ of $\mathcal{V}_{k+1}$, preconpact open neighborhoods $\mathcal{U}'''_{k+1}\subset\subset \mathcal{U}''_{k+1}  $ of $p_{k+1}$ in $X\sq \bar{\mathcal{W}}_{i,k}$, such that $ \mathcal{V}_{k+1} \sq \mathcal{V}'_{k+1} $ is compact, and $( \mathcal{V}'_{k+1} ,g_{\mathcal{V}} ) $ is isometric to $ 
( ( \mathcal{U}''_{k+1} \sq \mathcal{U}'''_{k+1} ) \times \sphere^2 ,g_{X,i,k} +\delta^2 f_{i,k}^2 g_{\sphere^2} ) $, where $\delta\in (0,1)$ is a constant.

By Corollary \ref{corollarymodify2}, one can obtain $f_{i, k+1}$ by modifying $\delta f_{i,k}$ on a small open neighborhood of $\mathcal{Y}_{i,k}$ such that the completion of $g_{X,i,k} + f_{i,k +1}^2 g_{\sphere^2} $ gives a smooth Riemannian metric around $\mathcal{Y}_{i,k}$. Then we can construct $(M_{i,k+1} ,g_{i,k+1})$ by gluing $(\mathcal{V}_{k+1} ,g_{\mathcal{V}} ) $ to $ 
\left( M_{i,k} \sq \left( \mathcal{U}'''_{k+1} \times \sphere^2 \right) ,g_{X,i,k} +\delta^2 f_{i,k}^2 g_{\sphere^2} \right) $. By applying Lemma \ref{lmmpuncture} again, one can obtain $\mathcal{U}'_{k+1}$, $\mathcal{Z}'_{k+1}$, $g_{X,i,k+1}$, $f_{i,k+1}$ and $F_{i,k+1}$ satisfying the conditions 1)${}_{i,k}$-10)${}_{i,k}$.

 Note here that we can ensure that for any $p\in \{p_0 \} \bigsqcup  \left(  \bigsqcup\limits_{i=0}^\infty \partial B^h_{2^{-i}} (p_0) \right) \bigsqcup \partial \mathcal{W}_{i,k} $ and $\sigma \in (0,1) $, the Gromov-Hausdorff distance
 $$ \sigma^{-1} d_{GH} (B_\sigma^{g_{X,i,k}} (p) , B_\sigma^{g_{X,i,k+1}} (p) ) \leq 2^{-101(i+k)} $$
 by choosing sufficiently small $r^* ,\delta >0$ when we applying Lemma \ref{lmmpuncture}. Similarly, by choosing sufficiently small $r^* ,\delta >0$, one can obtain that the identity map of $B^h_{2^{-2-i}} (p_0) $ is also $(1+\Psi (r^* ,\delta | g_{X,i} ) )$-bi-Lipschitz. 

 Now we finish the induction process about $k$.

Let $(M_{i+1} ,g_{i+1},g_{X,i+1} ,f_{i+1} ) = (M_{i,N'_i} ,g_{i,N'_i} ,g_{X,i,N'_i } ,f_{i, N'_i } ) $, $\mathcal{W}_{i+1} = \mathcal{W}_{i,N'_i} $, $\mathcal{Y}_{i+1} = \mathcal{Y}_{i,N'_i} $ and $F_{i+1} = F_{i,N'_i} $. Thus we can obtain the following property.

\begin{prop}\label{propnonmfd}
Let $(X , h)$ be an $n$-dimensional smooth complete Riemannian manifold with $n\geq 3$ and $\Ric_h > \lambda \in\mathbb{R} $. Fix $p_0 \in X $ and $\epsilon >0$. Then we have:
\begin{itemize}
\item A sequence of numbers $c_i \to 0$ as $i\to\infty$,
\item A sequence of $(n+2)$-dimensional Riemannian manifold $ (M_i , g_i ) $ with $\Ric_{g_i} > \lambda $,
\item An increasing sequence of open subsets 
$$\mathcal{W}_i = \bigsqcup\limits_{j=1}^{N_i} \mathcal{U}_{j} \subset \bigsqcup\limits_{j=-i}^{i} B^h_{2^{j+1} - 2^{-3i} } (p_0) \sq \bar{B}^h_{2^{j} + 2^{-3i} } (p_0) \subset (X,h) ,$$
\item A sequence of $n $-dimensional compact submanifolds $\mathcal{Y}_i = \bigsqcup_{j=1}^{N_i} \mathcal{Z}_{j} \subset M_i $,
\item A sequence of Riemannian metrics $ g_{X,i} $ and a sequence of positive functions $f_i$ on $X \sq \mathcal{W}_i $,
\item A sequence of isometries $F_i : ( M_i \sq \mathcal{Y}_i ,g_k ) \to \left( (X \sq \bar{\mathcal{W}}_i ) \times \sphere^2 , g_{X,i } + f_i^2 g_{\sphere^2 } \right) $,
\end{itemize}
such that:
\begin{itemize}
\item The functions $|f_i| \leq 2^{-i} $,
\item The closures $\bar{\mathcal{U}}_{j} $ are diffeomorphic to the closed unit $n$-ball,
\item The completions of $\left( X \sq \bar{\mathcal{W}}_i , d_{g_{X,i}} \right) $ are smooth Riemannian manifolds with boundary $\partial {\mathcal{W}_i}$,
\item The metrics $g_{X,i } = h $ on 
$$ B^h_{2^{-i}} (p_0) \cup \left( X \sq B^h_{2^{i+1}} (p_0) \right) \cup \bigsqcup\limits_{j=-i}^{i} B^h_{2^{j} + 2^{-4i} } (p_0) \sq \bar{B}^h_{2^{j} - 2^{-4i} } (p_0) \subset (X,h) ,$$
\item The Gromov-Hausdorff distance $d_{GH} \left( ( X \sq \bar{\mathcal{W}}_i )  , d_{g_{X,i }} ) , (X,h) \right) \leq \left( 1-2^{- i} \right) \epsilon $,
\item For any $ p\in \{p_0 \} \mathop{\cup}\limits_{j=0}^\infty \partial B^h_{2^{-j}} (p_0) $, the Gromov-Hausdorff distance 
$$r^{-1} d_{GH} \left( ( B^h_{r} (p) \sq \bar{\mathcal{W}}_i )  , d_{g_{X,i }} ) , B^{\mathbb{R}^n}_{r} (0) \right) \leq \Psi (r| n,X,h ) ,$$ 
\item For any $x\in B^{h}_{2^i} (p_0) \sq \mathcal{W}_i \subset \left( X \sq \mathcal{W}_i , d_{g_{X,i}} \right) $, the distance $d_{g_{X,i}} (x, \mathcal{W}_i ) \leq 2^{2-i} $,
\item For any $i\in\mathbb{N}$, the identity mapping $ (\partial W_i ,d_{g_{X,i+1}} ) \to (\partial W_i ,d_{g_{X,i}} ) $ is an $(1+2^{-i} )$-bi-Lipschitz map,
\item For any $i,j\in\mathbb{N}$ and $x\in ( X \sq \bar{\mathcal{W}}_i , d_{g_{X,i }} )$, there exists a closed subset $K_{x,i,j}$ of $\mathbb{R}^n$ satisfies that $ c_j^{-1} d_{GH} \left( B^{g_{X,i}}_{c_j } (x) , K_{x,i,j} \right) = \Psi (j^{-1} | d_h (x,p_0 ) , h )$,
\item For any positive integers $i\geq j $, $p\in \{p_0 \} \bigsqcup  \left(  \bigsqcup\limits_{i=0}^\infty \partial B^h_{2^{-i}} (p_0) \right) \bigsqcup \partial \mathcal{W}_j $, and $\sigma \in (0,1) $, the Gromov-Hausdorff distance
 $$ \sigma^{-1} d_{GH} (B_\sigma^{g_{X,i}} (p) , B_\sigma^{g_{X,j}} (p ) ) \leq 2^{-100j} .$$
\end{itemize}
\end{prop}

\section{Proof of Theorem \ref{thmnonmfd}}
\label{proofsection}

Now we are ready to prove Theorem \ref{thmnonmfd} and Corollary \ref{coroboundary}. For convenience, we restate Theorem \ref{thmnonmfd} here.

\begin{thm} 
Let $(X , h)$ be an $n$-dimensional smooth complete Riemannian manifold with $n\geq 3$ and $\Ric_h > \lambda \in\mathbb{R} $. Fix $\epsilon >0$. Then the sequence of $(n+2)$-dimensional manifolds $(M_{i} ,g_i ) \stackrel{GH}{\longrightarrow} (X_\epsilon ,d_\epsilon )$ we constructed in Proposition \ref{propnonmfd} satisfies the following properties.
\begin{enumerate}[(a).]
\item $d_{GH} (X,X_\epsilon ) \leq \epsilon $,
\item The Ricci curvature $\Ric_{g_i} > \lambda $,
\item $(X_\epsilon ,d_\epsilon ) $ is $n$-rectifiable,
\item For any open subset $U$ of $X_\epsilon $, the local $(n-1)$-homotopy group ${\rm{im}} \left( \pi_{n-1} (U ) \to \pi_{n-1} (X_\epsilon ) \right) $ and local $(n-1)$-homology group ${\rm{im}} \left( H_{n-1} (U;\mathbb{Z}) \to H_{n-1} (X_\epsilon ;\mathbb{Z}) \right) $ are infinitely generated. Consequently, every open subset of $X_\epsilon $ is not a topological manifold.
\end{enumerate}
\end{thm}

\begin{proof}
By the construction, one can easily to see that the sequence of $(n+2)$-dimensional manifolds $(M_{i} ,g_i ) \stackrel{GH}{\longrightarrow} (X_\epsilon ,d_\epsilon )$ satisfies (a) and (b).

By Proposition \ref{propnonmfd}, for any $x\in X_\epsilon$, there exist a sequence of constants $c_k\to 0$ and a sequence of closed subsets of $\mathbb{R}^n$, such that the Gromov-Hausdorff distance
$$\lim_{k\to\infty} c_k^{-1} d_{GH} \left( B^{g_{X,i}}_{c_k } (x) , K_{x,k} \right) = 0 .$$
It follows that the rectifiable dimension of $X_\epsilon \leq n$. See \cite[Corollary 1.36]{chco4} and \cite[Theorem 1.18]{coldingnaber1} for more details.

Since $$r^{-1} d_{GH} \left( ( B^h_{r} (p_0) \sq \bar{\mathcal{W}}_i )  , d_{g_{X,i }} ) , B^{\mathbb{R}^n}_{r} (0) \right) \leq \Psi (r| n,X,h ) ,$$
we see that $p_0 \in \mathcal{R}_{n } (X_\epsilon )$, and hence $R_n (X_\epsilon ) \neq \emptyset $. Thus the rectifiable dimension of $X_\epsilon \geq n$. Hence $(X_\epsilon ,d_\epsilon ) $ is $n$-rectifiable. This proves the property (c).

Now we consider the property (d). Let $U$ be an open subset of $X_\epsilon $. Assume that the homology group $H_{n-1} (U)$ is finitely generated. Without loss of generality, we can assume that 
$$\mathrm{rank} ( \mathrm{im} (H_{n-1} (U) \to H_{n-1} (X_\epsilon ) )) = N_U \in\mathbb{N} .$$ 

Let $p_\infty \in U $ and $r_\infty >0$ such that $B_{2r_\infty } ( p_\infty ) \subset U $. Choosing a sequence of points $p_i \in M_i$ converging to $p_\infty $. Then for any $N\in\mathbb{N}$, we can find $i_N \in\mathbb{N} $ such that for any $i\geq i_N$, the set $\mathcal{W}_{i}$ has at least $N $ connected components contained in $B_{2r_\infty } ( p_i ) \subset U $. Denote them by $\mathcal{U}_j $, $j=1,\cdots, N $.

Now we recall that the inclusion map $I_{i,i-1} : X \sq \bar{\mathcal{W}}_{i} \to X \sq \bar{\mathcal{W}}_{i-1}$ is an $(1+2^{-i} \epsilon)$-Lipschitz map, and induces an $(1+2^{-i} \epsilon)$-bi-Lipschitz map on the connected components $\partial \mathcal{U}_j \to \partial \mathcal{U}_j $ of the boundary of $\partial \mathcal{W}_{i} $. Then we can see that the inclusion maps $ I_{i_N +1,i_N } \circ \cdots \circ I_{i-1,i-2} \circ I_{i ,i-1}$ converge to a Lipschitz map $X_\epsilon \to X \sq \mathcal{W}_{i_N} $, and induce a bi-Lipschitz map between $\partial \mathcal{W}_{i_N}$ and a compact subset $\mathcal{Y}_{i_N}^{\infty} $ of $U\subset X_\epsilon $. Hence for $j=1,\cdots ,N$, the image of $H_{n-1} (U ) \to H_{n-1} (X \sq \mathcal{W}_{i_N} ) $ induced by $U\subset X_\epsilon \to X \sq \mathcal{W}_{i_N} $ contains the image of $\bigoplus_{j=1}^N H_{n-1} (\partial \mathcal{U}_{j} ) = H_{n-1} (\partial\mathcal{W}_{i_N} ) \to H_{n-1} (X \sq \mathcal{W}_{i_N} ) $ induced by the inclusion map $\bigsqcup_{j=1}^N \partial \mathcal{U}_{j} = \partial \mathcal{W}_{i_N} \to X \sq \mathcal{W}_{i_N} $. On the other hand, we have 
$$\mathrm{rank} \left( \mathrm{im} ( H_{n-1} (U ) \to H_{n-1} (X \sq \mathcal{W}_{i_N} ) ) \right) \leq \mathrm{rank} ( \mathrm{im} (H_{n-1} (U) \to H_{n-1} (X_\epsilon ) )) \leq N_U .$$
Hence we have
\begin{eqnarray*}
\mathrm{rank} \left( \mathrm{im} \left( \bigoplus_{j=1}^N H_{n-1} (\partial \mathcal{U}_{j} ) \to H_{n-1} (X \sq \mathcal{W}_{i_N} ) \right) \right) & \leq & \mathrm{rank} \left( \mathrm{image} ( H_{n-1} (U ) \to H_{n-1} (X \sq \mathcal{W}_{i_N} ) ) \right) \\
& \leq & N_U .
\end{eqnarray*}

Now we consider the quotient map $X\to X/\thicksim $, where the equivalent relation $\thicksim $ is to consider the complement of $V$ as a point, and $V$ is an open neighborhood of $B_{2r_\infty } ( p_i ) $ with smooth boundary. Then $X/\thicksim$ is a $n$-dimensional finite cell complex, and hence $ H_{n } (X/\thicksim )  $ is a finitely generated abelian group \cite[Lemma 2.34]{aha1}. By a similar argument as above, one can conclude that
$$\mathrm{rank} \left( \mathrm{im} \left( \bigoplus_{j=1}^N H_{n-1} (\partial \mathcal{U}_{j} ) \to H_{n-1} ((X/\thicksim ) \sq \mathcal{W}_{i_N} ) \right) \right) \leq N_U .$$
By the Mayer-Vietoris sequence \cite[Section 2.2]{aha1} given by $\cup_{j=1}^N \bar{\mathcal{U}}_{j} $ and $(X/\thicksim ) \sq \mathcal{W}_{i_N} $, we can get the following exact sequence:
$$ H_{n } (X/\thicksim ) \to \bigoplus_{j=1}^N H_{n-1} (\partial \mathcal{U}_{j} ) \to H_{n-1} ((X/\thicksim ) \sq \mathcal{W}_{i_N} ) \bigoplus H_{n-1} \left( \cup_{j=1}^N \bar{\mathcal{U}}_{j} \right) \to H_{n-1} (X/ \thicksim ) .$$
It follows that 
\begin{eqnarray*}
 N & = & \mathrm{rank} \left( \bigoplus_{j=1}^N H_{n-1} (\partial \mathcal{U}_{j} ) \right) \\
 & \leq & \mathrm{rank} \left( \mathrm{im} \left( \bigoplus_{j=1}^N H_{n-1} (\partial \mathcal{U}_{j} ) \to H_{n-1} ((X/\thicksim ) \sq \mathcal{W}_{i_N} ) \right) \right) + \mathrm{rank} H_{n } (X/\thicksim ) \\
 & \leq & N_U + \mathrm{rank} H_{n } (X/\thicksim )     
\end{eqnarray*}
But obviously an arbitrarily large $N$ can be taken here, a contradiction. It follows that the local $(n-1)$-homology group ${\rm{im}} \left( H_{n-1} (U;\mathbb{Z}) \to H_{n-1} (X_\epsilon ;\mathbb{Z}) \right) $ is infinitely generated. Similarly, one can see that the local $(n-1)$-homotopy group ${\rm{im}} \left( \pi_{n-1} (U ) \to \pi_{n-1} (X_\epsilon ) \right) $ is also infinitely generated. This proves (d), and the theorem follows.
\end{proof}

Then we prove Corollary \ref{coroboundary}.

\vspace{0.2cm}

\noindent \textbf{Proof of Corollary \ref{coroboundary}: }
Let $(M_{i} ,g_i ) \stackrel{GH}{\longrightarrow} (X_\epsilon ,d_\epsilon )$ be the sequence of $(n+2)$-dimensional manifolds we constructed in Proposition \ref{propnonmfd}. Then Proposition \ref{propnonmfd} implies that for any $x\in \mathcal{Y}_{j}^{\infty} $, all tangent cone at $x$ are isometric to $\mathbb{R}^{n-1} \times \mathbb{R}_{\geq 0} $, where $\mathcal{Y}_{j}^{\infty} $ is the limit of $\mathcal{W}_j$ in $X_\epsilon $. Hence $\cup_{j=1}^\infty \mathcal{Y}_{j}^{\infty} \subset \partial^* X_\epsilon $ is dense in $X_\epsilon $. Thus $\partial_{\rm DPG} X_\epsilon = X_\epsilon $. Applying Proposition \ref{propnonmfd} again, one can conclude that $\cup_{i=1}^{\infty} \partial B^h_{2^{-i}} \subset R_{n} (X_\epsilon )  $. Then $\cup_{i=1}^{\infty} ( \partial B^h_{2^{-i}} \cap \partial^* X_\epsilon ) =\emptyset $, and hence $\partial_{\rm KM} X_\epsilon$ has an infinite number of connected components.
\qed

\end{document}